\documentclass{amsart}
\usepackage{todonotes}
\newtheorem{lemma}{Lemma}
\newtheorem{theorem}{Theorem}
\newtheorem{corollary}{Corollary}

\usepackage[linktocpage]{hyperref}

\title{Stable Solutions to Generalizations of Fibonacci Relations of Higher Order}
\author{Timmy Ma, Patrick Vernon}

\begin{document}

\begin{abstract}\label{sect:abstract}
   We reformulate and discuss the conditions for the initial values in which a generalized Fibonacci sequence is eventually strictly monotonic. We then extend this exploration to the generalized \(k\)-bonacci sequence. We also give a lower bound on when an eventually strictly monotone generalized \(k\)-bonacci sequence becomes monotone.
\end{abstract}
\maketitle

\section{Introduction}\label{sect:intro}
The classical \textit{Fibonacci sequence, $\{f_n\}$,} is defined by: 
\begin{equation}\label{eqn:Fibonacci}
    \begin{cases}
        f_0=0,\\
        f_1=1,\\
        f_{n+1} = f_n + f_{n-1}.
    \end{cases}
\end{equation}  The first few terms of this sequence are as follows:
\[0, 1, 1, 2, 3, 5, 8, 13, \dots\]

The Binet formula for this sequence can be used to calculate the \(n\)th term \cite{vorobiev2002fibonacci}: 
\begin{equation}\label{eqn:binet_fib}
    f_n = \frac{1}{\sqrt{5}}\varphi^n - \frac{1}{\sqrt{5}}\psi^n,
\end{equation} 
where \(\varphi,\psi\) are the solutions to the equation \(x^2 - x - 1 = 0\); specifically, \[\varphi = \frac{1 + \sqrt{5}}{2} \approx 1.618\] and \[\psi = -\frac{1}{\varphi} = \frac{1 - \sqrt{5}}{2} \approx -0.618.\] 
Note that \(\frac{1}{\sqrt{5}}\psi^n\) is a vanishing term of the Binet formula since \(|\psi|<1\), meaning \(\left|f_n - \frac{1}{\sqrt{5}}\varphi^n\right|\) approaches \(0\) as \(n\) approaches infinity. Additionally, this means that the ratio \(f_{n+1} / f_n\) approaches \(\varphi\) as \(n\) approaches infinity. This is referred to as the golden ratio. 

Similarly to the Fibonacci sequence, we define the \textit{Tribonacci \(\{T_n\}\) sequence }as the solution to: 
\begin{equation}\label{eqn:trib}
    \begin{cases}
        T_0=T_1=0,\\
        T_2=1,\\
        T_{n + 1} = T_n + T_{n - 1} + T_{n - 2} .
    \end{cases}
\end{equation} 
The Binet formula for this sequence is:
\begin{equation}\label{eqn:binet-trib}
T_n = \frac{\rho^n}{(\rho - \sigma)(\rho - \bar{\sigma})} + \frac{\sigma^n}{(\sigma - \rho)(\sigma - \bar{\sigma})} + \frac{\bar{\sigma}^n}{(\bar{\sigma} - \rho)(\bar{\sigma} - \sigma)},
\end{equation}
where \(\rho, \sigma,\) and \(\bar{\sigma}\) are the solutions to the complex equation \(z^3 - z^2 - z - 1 = 0\) \cite{spickerman1982binet}. The approximate values are \[\rho \approx 1.8393, \sigma \approx -0.4196 + 0.6063i, \text{ and } \bar{\sigma} \approx -0.4196 - 0.6063i.\] Again, only one of these solutions has magnitude at least \(1\), so the Binet formula for \(T_n\) has only one non-vanishing term, the term involving \(\rho^n\).

We can also consider sequences that are solutions to linear recurrence relations of order \(k\) for some \(k \geq 2\), such that all coefficients are \(1\), \cite{miles1960generalized}. That is, for any integer \(k\geq 2\), the \textit{\(k\)-bonacci sequence} \(\{q_n\}\) is defined as the solution to:
\begin{equation}\label{eqn:k-bon}
    \begin{cases}
        q_1=q_2=\dots=q_{k-2}=0,\\
        q_{k-1}=1,\\
        q_{n + 1} = q_n+ q_{n-1}+\dots+q_{n-k+1}.
    \end{cases}
\end{equation} Note that in the cases when \(k = 2\) and when \(k = 3\) we have the classical Fibonacci (\ref{eqn:Fibonacci}) and Tribonacci (\ref{eqn:trib}) sequences described above.

A \textit{generalized \(k\)-bonacci sequence}, $\{F_n\}$, is defined as a solution to the relation:
\begin{equation}\label{eqn:gen_k-bon}
F_{n+1} = F_{n} + \cdots + F_{n - k+1}.
\end{equation} This formulation is discussed in \cite{dresden2014simplified, ferguson1966expression, flores1967direct, gabai1970generalized,lee2001binet,miles1960generalized, spickerman1984binet, verma2025mathematics, wolfram1998solving}. When $k=2$, we call this a generalized Fibonacci sequence. When $k=3$, we call this a generalized Tribonacci sequence.

A sequence \(\{a_n\}\) is \textit{eventually strictly monotone increasing} if there exists \(K\) such that for all \(n > K\) we have \(a_{n+1} > a_n\). We can define \textit{strictly monotone decreasing} similarly. Certainly, we can see that the classical Fibonacci sequence is increasing.  In \cite{ozban2025eventually}, the authors studied restrictions on the initial conditions in the case where $k=2$, which resulted in an eventually strictly monotone increasing (decreasing) sequence. We can also explore the conditions on the initial values which allow (\ref{eqn:gen_k-bon}) to be eventually strictly monotone. Note that since each term in (\ref{eqn:gen_k-bon}) is the sum of its $k$ previous terms, the property of being eventually strictly monotone is equivalent to the sequence having $k$ consecutive terms of the same sign. We will also define a generalized Fibonacci sequence to be \textit{stable} if it is not eventually strictly monotone.

In section \ref{sect:background}, we provide a background description to set up our discussions and a restatement of the results from \cite{ozban2025eventually}. 

In section \ref{sect:gen_k-bon}, we discuss the $k$-bonacci sequence, the generalized $k$-bonacci sequence, and their Binet formula. We then discuss in section \ref{sect:suff-cond} the conditions for the generalized $k$-bonacci sequence to be either stable or eventually strictly monotone.

\section{Background}\label{sect:background}


We make a quick observation about the Binet formula of the generalized Fibonacci sequence. The proof is left to the reader.

\begin{lemma}
\label{general-relate-to-conventional}
    Let \(\{F_n\}\) be a sequence such that \(F_{n + 1} = F_n + F_{n - 1}\) for all \(n\). Then \[F_n = f_{n - 1} F_0 + f_n F_1\] for all \(n\), where $f_n$ are the terms of the classical Fibonacci sequence, (\ref{eqn:Fibonacci}).
\end{lemma}

A direct result of Lemma \ref{general-relate-to-conventional} is the following:
\begin{eqnarray*}
    F_n &=& f_{n - 1} F_0 + f_n F_1\\
    &=& \frac{1}{\sqrt{5}} \Big(\left(\varphi^{n - 1} + \psi^{n - 1}\right) F_0 + \left(\varphi^n + \psi^n\right) F_1 \Big).
\end{eqnarray*}

Thus, we have a formulation of the \textbf{Binet Formula for the generalized Fibonacci sequence},
\begin{eqnarray}\label{Binet-general-fibonacci}
    F_n &=& \frac{1}{\sqrt{5}}\Big(\varphi^{n - 1}\left(F_0 + \varphi F_1\right) + \psi^{n - 1}\left(F_0 + \psi F_1\right)\Big).
\end{eqnarray}
Other such formulations can be found in \cite{horadam1961generalized, mahajan2014binet, miles1960generalized, verma2025mathematics}.

\subsection{Eventually strictly monotone}\label{sect:strict_monotone}

 We now provide a restatement of the previous authors' results \cite{ozban2025eventually}.

\begin{theorem}\label{thm:fib_mono}
    Let \(\{F_n\}\) be a sequence such that \(F_{n + 1} = F_n + F_{n - 1}\) for all \(n\). If \(F_1 = (-\frac{1}{\varphi})F_0,\) then \(\{F_n\}\) is stable. If \(F_1 > (-\frac{1}{\varphi})F_0 ,\) then \(\{F_n\}\) is eventually strictly monotone increasing, and if \(F_1 < (-\frac{1}{\varphi})F_0 ,\) then \(\{F_n\}\) is eventually strictly decreasing.
\end{theorem}

\begin{proof}
    This follows immediately from the fact that there is only one possible non-vanishing term in the Binet formula given by (\ref{Binet-general-fibonacci}), namely, \(\frac{1}{\sqrt{5}}\varphi^{n - 1}(F_0 + \varphi F_1)\). As \(n\) approaches infinity, then \(F_n\) will approach infinity if \(F_0 + \varphi F_1 > 0\), will approach negative infinity if \(F_0 + \varphi F_1 < 0\), and will approach zero if \(F_0 + \varphi F_1 = 0\).
\end{proof}

\begin{corollary}
    Any stable generalized Fibonacci sequence \(\{F_n\}\) approaches zero as \(n\) approaches infinity, and the ratio \(F_{n+1} / F_n\) for any stable generalized Fibonacci sequence is equal to \(\psi\) for any \(n\).
\end{corollary}
\begin{proof}
    Since the generalized Fibonacci sequence is stable then $F_1=(-\frac{1}{\varphi})F_0$. Then we have that \[\frac{F_{n+1}}{F_n} = \frac{\psi^n(F_0+\psi F_1)}{\psi^{n-1}(F_0+\psi F_1)}=\psi\] for all $n$.
\end{proof}
\begin{corollary}
    For any eventually monotone generalized Fibonacci sequence \(\{F_n\}\), the ratio \(F_{n + 1} / F_n\) approaches \(\varphi\) as \(n\) approaches infinity.
\end{corollary}
\begin{proof}
    Since the generalized Fibonacci sequence is eventually monotone then we have that $F_1\not=(-\frac{1}{\varphi})F_0$, and $\psi^{n-1}(F_0+\psi F_1)$ is the vanishing term of (\ref{Binet-general-fibonacci}), thus we have 
    \[\frac{F_{n+1}}{F_n} \rightarrow  \frac{\varphi^n(F_0+\varphi F_1)}{\varphi^{n-1}(F_0+\varphi F_1)}=\varphi\] as $n$ approaches infinity.
\end{proof}

\section{Generalized k-bonacci sequence}\label{sect:gen_k-bon}

We begin by stating the Binet formula for the \(k\)-bonacci sequence. 

\subsection{Binet formula of \texorpdfstring{$k$}{k}-bonacci sequence}\label{subsect:k-bon}
Many results are known about the Binet formula for the $k$-bonacci sequence, \cite{lee2001binet, miles1960generalized, miller1971generalized, spickerman1984binet}. In particular, we have the Binet formula for the \(k\)-bonacci sequence (\ref{eqn:k-bon}):
\begin{equation}\label{eqn:binet_k-bon}
    q_n = \sum_{j = 1}^k A_j \alpha_j^n
\end{equation} for each \(n\), and \[A_j=\frac{1}{\prod_{i\not=j}^k(\alpha_j-\alpha_i)}.\]
\begin{itemize}
    \item \(\frac{1}{\alpha_j}\) are the roots of \(g_k(x)=1-x-x^2-\dots-x^k\) for $j=1,\dots,k,$
    \item \(|\alpha_j|<1\) for all but one, without loss of generality, assume \(|\alpha_1|>1\).
    \item \(\alpha_j\not=\alpha_l\) if \(j\not=l\).
\end{itemize}
Similarly, from (\ref{general-relate-to-conventional}), we can express the generalized \(k\)-bonacci sequence as:
    \begin{equation}\label{eqn:Fn_qn}
        F_n = \sum_{m = 0}^{k - 1} \left(q_{n - m - 1} \sum_{j = m}^{k - 1} F_j\right).
    \end{equation}

\subsection{Binet formula of generalized \texorpdfstring{$k$}{k}-bonacci}\label{subsect:gen_k-bon}
Using (\ref{eqn:binet_k-bon}) and (\ref{eqn:Fn_qn}), and rearranging of the terms we have 
the \textbf{Binet formula of the generalized \(k\)-bonacci sequence}:
\begin{equation}\label{eqn:binet_gen-k}
    F_n= \sum_{j = 1}^k A_j \alpha_j^{n - k}\left( (\alpha_j ^{k - 1})F_0 + (\alpha_j^{k - 1} + \alpha_j^{k - 2}) F_1 + \cdots + (\alpha_j^{k - 1} + \cdots + 1)F_{k - 1}\right).
\end{equation}

The first term in the formula for \(q_n\) is the only non-vanishing term since $|\alpha_j|<1$ for each $j\not=1$. Thus \(q_n\) approaches \(A_1 \alpha_1^n\) as \(n\) approaches infinity. Since the sequence \(\{q_n\}\) contains only positive terms, when $n\geq 2$, then $A_1$ must be real and positive. Specifically, following the steps from \cite{miller1971generalized, spickerman1984binet}, \[A_1=\frac{\alpha_1^2-\alpha_1}{2\alpha_1^k-(k+1)}.\] 

\begin{theorem}\label{thm:k-bon_mono}
    Let \(\{F_n\}\) be a sequence such that \[F_{n+1}=F_n+F_{n-1}+\dots+F_{n-k+1}\] for all \(n\). Let $\frac{1}{\alpha_1}$ be the unique root of \(g_k(x)=1-x-x^2-\dots-x^k\) such that \(|\alpha_1|>1\). Then \(\{F_n\}\) is stable if and only if 
    \begin{equation}\label{eqn:stable_cond_gen_k-bon}
        (\alpha_1 ^{k - 1})F_0 + (\alpha_1^{k - 1} + \alpha_1^{k - 2}) F_1 + \cdots + (\alpha_1^{k - 1} + \cdots + 1)F_{k - 1} = 0.
    \end{equation}
\end{theorem}
\begin{proof}
Note that from the formulation of \(g_k(x)\), we find that \(\alpha_j^n\) approaches \(0\) for each \(j > 1\) for large $n$.
 Thus, from (\ref{eqn:binet_gen-k}), we have \[|F_n - A_1 \alpha_1^{n - k}\left( (\alpha_1 ^{k - 1})F_0 + (\alpha_1^{k - 1} + \alpha_1^{k - 2}) F_1 + \cdots + (\alpha_1^{k - 1} + \cdots + 1)F_{k - 1}\right) |\rightarrow 0\] as \(n\) approaches infinity.
 Since \(A_1\) is real and positive, and \(\alpha_1\) has magnitude greater than 1, then \(\{F_n\}\) is stable if and only if \[ (\alpha_1 ^{k - 1})F_0 + (\alpha_1^{k - 1} + \alpha_1^{k - 2}) F_1 + \cdots + (\alpha_1^{k - 1} + \cdots + 1)F_{k - 1} = 0\]
 \end{proof}
 \begin{corollary}
     For any eventually monotone generalized \(k\)-bonacci sequence, $\{F_n\}$, the ratio $F_{n+1}/F_n$ approaches $\alpha_1$ as $n$ approaches infinity.
 \end{corollary}
 \begin{proof}
     Since $\{F_n\}$ is eventually monotone, then (\ref{eqn:stable_cond_gen_k-bon}) does not hold. Since $\alpha_j^n$ is the vanishing term of (\ref{eqn:binet_gen-k}) for each $j>1$, we find that the ratio 
     \[\frac{F_{n+1}}{F_n} \rightarrow \frac{A_1\alpha_1^{n-k+1}\left( (\alpha_1 ^{k - 1})F_0  + \cdots + (\alpha_1^{k - 1} + \cdots + 1)F_{k - 1}\right)}{A_1\alpha_1^{n-k}\left( (\alpha_1 ^{k - 1})F_0 + \cdots + (\alpha_1^{k - 1} + \cdots + 1)F_{k - 1}\right)} = \alpha_1\] as $n$ approaches infinity.
 \end{proof}
\section{Sufficient bounds}\label{sect:suff-cond}
Given an eventually strictly monotone $k$-bonacci sequence with initial conditions $F_0, F_1,\dots, F_{k-1}$, we find $N$ such that for all $n>N$, we have $F_n$ are of the same sign.
 \begin{theorem}\label{thm:suff_cond}
     Let $\{F_n\}$ be an eventually strictly monotone increasing (decreasing) $k$-bonacci sequence. 
    Let \[N = k+\log_{\alpha_1}\left(\frac{\sum_{m = 0}^{k - 1}\left|\sum_{j = m}^{k - 1}F_j\right|}{2 A_1 \sum_{m = 0}^{k - 1}\left(\alpha_1^{k - m - 1}\sum_{j = m}^{k - 1}F_j\right)}\right), \] then for all $n>N$, we have $F_n>0$ ($F_n<0$).
 \end{theorem}
 \begin{proof}
 Note that from \cite{dresden2014simplified}, for any \(k\) and any \(n\), the distance between \(q_n\) and \(A_1 \alpha_1^n\) is less than \(1 / 2\).
 
 From (\ref{eqn:binet_k-bon}) and (\ref{eqn:Fn_qn}), we have the following:
     \begin{align*}
        F_n &= \sum_{m = 0}^{k - 1}\left(q_{n - m - 1} \sum_{j = m}^{k - 1}F_j\right)\\
        &= \sum_{m = 0}^{k - 1}\left(\left(\sum_{i = 1}^k A_i \alpha_i^{n - m - 1}\right)\left(\sum_{j = m}^{k - 1}F_j\right)\right)\\
        &= \sum_{m = 0}^{k - 1}\left(A_1 \alpha_1^{n - m - 1}\right)\left(\sum_{j = m}^{k - 1}F_j\right)
        + \sum_{m = 0}^{k - 1}\left(\left(\sum_{i = 2}^k A_i \alpha_i^{n - m - 1}\right)\left(\sum_{j = m}^{k - 1}F_j\right)\right)\\
        &= \sum_{m = 0}^{k - 1}\left(A_1 \alpha_1^{n - m - 1}\right)\left(\sum_{j = m}^{k - 1}F_j\right)
        + \sum_{m = 0}^{k - 1}\left(\left(q_{n - m - 1} - A_1 \alpha_1^{n - m - 1}\right)\left(\sum_{j = m}^{k - 1}F_j\right)\right)\\
        &= A_1 \alpha_1^{n - k}\sum_{m = 0}^{k - 1}\left(\alpha_1^{k - m - 1}\sum_{j = m}^{k - 1}F_j\right)
        + \sum_{m = 0}^{k - 1}\left(\left(q_{n - m - 1} - A_1 \alpha_1^{n - m - 1}\right)\left(\sum_{j = m}^{k - 1}F_j\right)\right)\\
        &> A_1 \alpha_1^{n - k}\sum_{m = 0}^{k - 1}\left(\alpha_1^{k - m - 1}\sum_{j = m}^{k - 1}F_j\right)
        - \sum_{m = 0}^{k - 1}\left(\frac{1}{2} \left|\sum_{j = m}^{k - 1}F_j\right|\right)\\
        &= A_1 \alpha_1^{n - k}\sum_{m = 0}^{k - 1}\left(\alpha_1^{k - m - 1}\sum_{j = m}^{k - 1}F_j\right)
        - \frac{1}{2} \sum_{m = 0}^{k - 1}\left|\sum_{j = m}^{k - 1}F_j\right|\\
    \end{align*}

    Assuming the sequence diverges to infinity, we know that \(\sum_{m = 0}^{k - 1}\left(\alpha_1^{k - m - 1}\sum_{j = m}^{k - 1}F_j\right)\) must be positive. Then a \textbf{sufficient condition for \(F_n\) to be positive} is:

    \[n \ge k + \log_{\alpha_1}\left(\frac{\sum_{m = 0}^{k - 1}\left|\sum_{j = m}^{k - 1}F_j\right|}{2 A_1 \sum_{m = 0}^{k - 1}\left(\alpha_1^{k - m - 1}\sum_{j = m}^{k - 1}F_j\right)}\right)\]
 \end{proof}
 
    In general, it can be difficult to find exact values for each \(\alpha_m\) and each \(A_m\). However, the above expression is easy to estimate using approximations since it depends only on \(\alpha_1\) and \(A_1\) along with values that are given to us. To estimate \(\alpha_1\), note that \(q_n / (A_1 \alpha_1^n)\) approaches \(1\) as \(n\) approaches infinity, so \(q_{n+1} / q_n\) approaches \(\alpha_1\). Then observe that \(q_n / \alpha_1^n\) approaches \(A_1\) as \(n\) approaches infinity. 

As an example, let \(k=4\), \(F_0=-3, F_1=-3.1, F_2=2, F_3=2.2\), then we have \(\alpha_1= 1.927561975\dots\), \(A_1=0.079078\dots\), and the following:\\
\begin{center}
\begin{tabular}{|c|c|c|}
\hline
    \(n\) & \(F_n\) & \(F_{n+1}/F_n\) \\
\hline
     0 & -3 & -- \\
     1 & -3.1 & 1.03333\dots \\
     2 & 2 & -0.64516\dots\\
     3 & 2.2 & 1.1\\
     4 & -1.9 & -0.86364\dots\\
     5 & -0.8 & 0.421053\dots\\
     6 & 1.5 & -1.875\dots\\
     7 & 1 & 0.6666\dots\\
     8 & -0.2 & -0.2\\
     9 & 1.5 & -7.5\\
     10 & 3.8 & 2.53333\dots\\
     11 & 6.1 & 1.605263\dots\\
     12 & 11.2 & 1.836066\dots\\
     13 & 22.6 & 2.017857\dots\\
     14 & 43.7 & 1.933628\dots\\
     15 & 83.6 & 1.913043\dots\\
     16 & 161.1 & 1.927033\dots\\
     17 & 311 & 1.930478\dots\\
     18 & 599.4 & 1.927331\dots\\
     19& 1155.1 &   1.927094\dots\\
     \vdots & \vdots & \vdots\\
\hline
\end{tabular}
\end{center}
According to Theorem \ref{thm:suff_cond}, we have that \(N=9.873909\), thus for all \(n>N\), \(F_n>0\), which we can see from the table above.

\
    
\bibliographystyle{plain}
\bibliography{bibfile}

\newpage
\end{document}